\documentclass[10pt]{article}

\usepackage[english]{babel}
\usepackage{amsmath}
\usepackage{amssymb}
\usepackage{amsthm}
\usepackage[T1]{fontenc}
\usepackage[utf8]{inputenc} 
\usepackage{verbatim}
\usepackage{enumerate}
\usepackage[final]{pdfpages}
\usepackage{epigraph}
\usepackage{geometry} 
\usepackage{mathtools}

\def\Y#1#2{\mathcal{Y}_{#1}^{#2}}
\newcommand{\eps}{\varepsilon}

\newtheorem{thm}{Theorem}
\newtheorem{lemma}[thm]{Lemma}
\newtheorem{claim}[thm]{Claim}
\newtheorem{alg}[thm]{Algorithm}

\newtheorem{prop}[thm]{Proposition}

\newtheorem{fact}[thm]{Fact}
\theoremstyle{remark}
\newtheorem*{remark}{Remark}
\newtheorem*{defn}{Definition}
\newtheorem{conj}[thm]{Conjecture}

\renewcommand\footnotemark{}

\title{Multi-coloured jigsaw percolation on random graphs}

\author
{Oliver Cooley, Abraham Guti\'errez$^{*}$
\footnote{Institute of Discrete Mathematics, Graz University of Technology, Steyrergasse 30, 8010 Graz, Austria.}
\footnote{$\lbrace$cooley, a.gutierrez$\rbrace$@math.tugraz.at}
\footnote{$^*$Supported by Austrian Science Fund project FWF P29355-N35.}
\\
\today
}

\date{}

\begin{document}

\maketitle

\begin{abstract}
The jigsaw percolation process, introduced by Brummitt, Chatterjee, Dey and Sivakoff, was inspired by
a group of people collectively solving a puzzle. It can also be seen as a measure of whether two graphs
on a common vertex set are ``jointly connected''. In this paper we consider the natural generalisation
of this process to an arbitrary number of graphs on the same vertex set. We prove that if these graphs
are random, then the jigsaw percolation process exhibits a phase transition in terms of the product of
the edge probabilities. This generalises a result of Bollob\'as, Riordan, Slivken and Smith.

\medskip

\noindent \emph{Mathematics Subject Classification:} 05C80
\end{abstract}

\section{Introduction}

\subsection{Jigsaw Percolation.}

In recent years there has been significant research inspired by the observation that certain advances
are only possible as a result of the collaboration of a group of people, rather than the work of one
individual e.g.~\cite{Ball14, Newman01b, Newman01a, Tebbe11}.

To model this mathematically, Brummitt, Chatterjee, Dey and Sivakoff~\cite{jigsawBCDS} introduced
the jigsaw percolation process. The premise is that a group of people each have one piece of a puzzle
which must be combined in a certain way to solve the puzzle. The individuals (and their associated
puzzle pieces) are represented by a set of vertices, and there are two graphs on these vertices:
a \emph{people graph}, with an edge if the two people know each other; and a \emph{puzzle graph}
with an edge if the two puzzle pieces are compatible. In the jigsaw percolation process, we begin with
each vertex forming its own \emph{cluster} and we merge
two clusters if there is an edge between them in both the people and the puzzle graph -- this represents these 
two people sharing all their information. The new merged cluster inherits all the incident edges 
of the original clusters. (The process will be described more formally
later.) This process is iterated until either there is only
one cluster remaining, in which case we say that the process \emph{percolates} indicating that the
puzzle has been solved, or no more clusters can be merged, in which case we say that the process
\emph{does not percolate}. More generally, if the two graphs are $G_1$ and $G_2$, we say that
the \emph{double-graph} $(G_1,G_2)$ percolates or does not percolate respectively.

This process was introduced by Brummitt, Chatterjee, Dey and Sivakoff in~\cite{jigsawBCDS}
 and was also considered by Gravner and Sivakoff in~\cite{GravnerSivakoff17}.

Bollob\'as, Riordan, Slivken and Smith~\cite{jigsawBRSS} considered the
case when the people graph and the puzzle graph are independent binomial random graphs, and proved
that the property of the two graphs percolating undergoes a phase transition in terms of the product
of the two associated edge probabilities.
More precisely, their result can be stated as follows. Let $G(n,p)$ denote the Erd\H{o}s-R\'enyi
binomial random graph on vertex set $[n]:=\{1,2,\ldots,n\}$ in which each pair of vertices
forms an edge with probability $p$ independently of each other.
We say that a property or event holds
\emph{with high probability} (abbreviated to \emph{whp}), if it holds with probability tending
to $1$ as $n$ tends to infinity.

\begin{thm}[\cite{jigsawBRSS}]\label{thm:JigsawBRSS}
	There exists a constant $c$ such that the following holds:
	let  $G_{1} = G(n, p_{1}), G_{2} = G(n, p_2)$ be independent binomial
	random graphs on the same vertex set,
	where $0\leq p_{1} = p_{1}(n), p_{2}= p_{2}(n)\leq 1$. Then
	\begin{itemize}
		\item[$(i)$] if $p_{1}p_{2}\leq \frac{1}{cn\ln n}$ or $\min \{ p_{1},p_{2} \} \leq \frac{\ln n}{cn}$ then whp $(G_{1},G_{2})$ does not percolate;
		\item[$(ii)$]if $p_{1}p_{2}\geq \frac{c}{n \ln n}$ and $\min \{ p_{1}, p_{2} \} \geq \frac{c \ln n}{n},$ then whp $(G_{1},G_{2})$ percolates. 
	\end{itemize}
\end{thm}

Note that this theorem is not quite stated as it appeared in~\cite{jigsawBRSS},
but it is easy to derive this form from the original.
We also observe that connectedness of each of the two graphs is a necessary (but not sufficient)
condition for percolation of the double-graph. The conditions on $\min \{ p_{1}, p_{2} \}$ determine
whether this necessary condition is satisfied whp, since the threshold for connectivity is at 
$p = \frac{\ln n}{n}$ as famously proved by Erd\H{o}s and R\'enyi in~\cite{connected}.

Theorem~\ref{thm:JigsawBRSS} was extended to hypergraphs by Bollob\'as, Cooley, Kang and Koch~\cite{hyperpercolation},
with a whole family of generalisations of the percolation process to $k$-uniform hypergraphs in which the clusters
consist of $j$-sets of vertices for $1\le j \le k-1$.

In this paper, we extend in a different direction, namely to an arbitrary number of graphs on the same vertex set.

\begin{defn}\label{def:foldgraph}
	An \textit{$r$-fold graph} is an $(r+1)$-tuple $\mathbf{G}$ $:=(V,E_{1}, ... , E_{r})$,
	where $V:= [n]$ is the set of vertices and $E_{i}\subseteq {\binom{V}{2}}$ for each $i \in [r].$
	We will call $1,2...,r$ the \textit{colours} of $\mathbf{G}$ and the graph $G_{i}=(V,E_{i})$ will
	be said to be of colour $i$ for every $i \in [r].$ 
\end{defn}

The multi-coloured jigsaw algorithm on an $r$-fold graph is
the natural generalisation of the $2$-coloured version in which clusters must be joined
by an edge of each colour in order to merge. A formal description of this algorithm is given
later in Algorithm~\ref{percolation:definition}.

It is easy to see that percolation for $r=1$ is equivalent to connectedness of the graph.
Thus, percolation of the jigsaw process is a generalised
notion of connectedness of multiple graphs on the same vertex set.
Therefore Theorem~\ref{thm:JigsawBRSS} and the main results of this paper.
(Theorems~\ref{MainResultConst} and~\ref{MainResultNonConst})
may be seen as generalisations of the connectedness threshold of Erd\H{o}s and R\'enyi~\cite{connected}.

\subsection{Main theorem}

To state the main result of the paper, we introduce the following generalisation
of the binomial model for random graphs.

\begin{defn}
	An \textit{$r$-fold binomial random graph} $\mathbf{G}(n, p_{1},...,p_{r})$ is an $r$-fold graph
	$([n], E_{1},...,E_{n})$ where \linebreak $([n],E_{i})\sim G(n, p_{i})$ are independent binomial
	random graphs for every $i \in [r].$
\end{defn} 

The following generalisation of Theorem~\ref{thm:JigsawBRSS} is the main result of this paper.

\begin{thm}\label{MainResultConst}
	Let $2 \leq r \in \mathbb{N}$. There exists a constant $C_{r}$ such that the following holds: suppose that
	$p_{1},...,p_{r}$ are functions of $n$ such that $ 0 \leq p_{1}\leq p_{2} \leq ... \leq p_{r}\leq 1$
	and $\mathbf{G} = \mathbf{G}(n, p_{1}, p_{2} ,..., p_{r})$. For $i\in [r]$ let \linebreak
	\textit{$P_{i}$ }$:= p_{1}p_{2}...p_{i}$. Then
	\begin{itemize}
		\item[$(i)$] if $P_{i} \leq \frac{1}{C_{r} n (\ln n)^{i-1}}$ for some $2 \leq i \leq r$ or
		$P_{1} \leq \frac{\ln n}{C_{r}n} $ then whp $\mathbf{G}$ does \emph{not} percolate;
		\item[$(ii)$]if $P_{i} \geq \frac{C_{r}}{n (\ln n)^{i-1}}$ for every $ 2\leq i\leq r $ and
		$P_{1} \geq \frac{C_{r} \ln n}{n},$ then whp  $\mathbf{G}$ percolates.
	\end{itemize}
\end{thm}

In fact we will prove a slightly stronger result; we allow $r$ tend to infinity sufficiently slowly
as a function of $n$
(then $C_{r}$ also depends implicitly on $n$).
\begin{thm}\label{MainResultNonConst}
	Let $2 \leq r = o(\sqrt{\ln \ln n})$ and $C_{r} := 2^{8r^{2}}$. Then the following holds:
	suppose that $p_{1},...,p_{r}$ are functions of $n$ such that
	$ 0 \leq p_{1}\leq p_{2} \leq ... \leq p_{r}\leq 1$ and $\mathbf{G} = \mathbf{G}(n, p_{1}, p_{2} ,..., p_{r})$.
	Then
	\begin{itemize}
		\item[$(i)$] if $P_{i} \leq \frac{1}{C_{r} n (\ln n)^{i-1}}$ for some $2 \leq i \leq r$ or
		$P_{1} \leq \frac{\ln n}{C_{r}n} $ then whp $\mathbf{G}$ does \emph{not} percolate;
		\item[$(ii)$]if $P_{i} \geq \frac{C_{r}}{n (\ln n)^{i-1}}$ for every $ 2\leq i\leq r $ and
		$P_{1} \geq \frac{C_{r} \ln n}{n},$ then whp  $\mathbf{G}$ percolates.
	\end{itemize}
\end{thm}
Note that both in the proof of Theorem~\ref{thm:JigsawBRSS} in~\cite{jigsawBRSS} and
in the proof of Theorems~\ref{MainResultConst} and~\ref{MainResultNonConst} in this paper,
no attempt is made to optimise the constants $c$ and $C_r$,
and the value given in Theorem~\ref{MainResultNonConst} is probably far from best possible.

\begin{remark}\label{FoldConnected}
	Given an $r$-fold graph $\mathbf{G}=([n], E_{1},...,E_{r})$ it is easy to see that
	percolation of every $i$-fold graph $([n], E_{j_1},...,E_{j_i})$ obtained by
	considering a subset of $i$ colours
	is a necessary condition for percolation of $\mathbf{G}$ (but not sufficient). For $i=1$,
	we guarantee connectedness by taking $p_{1} = P_{1} \geq \frac{C_{r} \ln n}{n}.$
	For $2 \leq i \leq r$
	the inequalities $P_{i}\geq \frac{C_{r}}{n(\ln n)^{i-1}}$ together with
	$p_{1}\leq p_{2} \leq ... \leq p_{r}$
	ensure that every such $i$-fold graph percolates whp.
\end{remark}

\begin{proof}[\bf{Proof of Theorem~\ref{MainResultConst}}]
	Theorem~\ref{MainResultConst} follows immediately from Theorem~\ref{MainResultNonConst}.	
\end{proof}

We will therefore focus on proving Theorem~\ref{MainResultNonConst}. We will only present the proof of the
supercritical case since the proof of the subcritical case is an obvious generalisation of the 
corresponding proof for 2 colours in \cite{jigsawBRSS}. It is a simple first moment argument which we
omit here, see \cite{Thesis} for details.

While much of the proof of the supercritical case follows that in~\cite{jigsawBRSS}, there are
important differences for the multi-coloured case which present additional difficulty.
We will point out these differences in the course of the proof.

\subsection{The multi-coloured jigsaw algorithm.}
The multi-coloured jigsaw process is formally described as follows

\begin{alg}[The multi-coloured jigsaw algorithm]\label{percolation:definition}
	\hspace{8cm} \newline Input: $r$-fold graph $\mathbf{G} :=([n],E_{1}, ... , E_{r})$.
	\newline At time $t \geq 0$ 	there is a partition $\mathcal{C}_{t} =\{C_{t}^{1}, C_{t}^{2}, ...., C_{t}^{k_{t}}\}$
	of the vertex set $[n],$ which we construct inductively as follows:
  	\begin{enumerate}
		\item We take $k_{0}= n$, set $C_{0}^{j}:=\{j\}$ and $\mathcal{C}_{0} =\{ \{1\},\ldots, \{n\}\}$ for all $j \in [n]$ i.e.\ we begin at time $0$ with the discrete partition into single vertices.
		
		\item At time $t \geq 0,$ construct a graph $H_{t}$ on vertex set $\mathcal{C}_{t}$ by joining
		$\mathcal{C}_{t}^{i}$ to          $\mathcal{C}_{t}^{j}$ if there exist edges $e_{s} := \{v_{i,s},v_{j,s} \} \in E_{i} $
		for all $s \in [r]$  such that  $v_{i,s} \in \mathcal{C}_{t}^{i}$ and $v_{j,s} \in \mathcal{C}_{t}^{j}.$
		
		\item If $E(H_{t}) = \varnothing, $ then STOP. Otherwise, construct the partition 
		$$
		\mathcal{C}_{t+1} = \{ C_{t+1}^{1},..., C_{t+1}^{k_{t+1}} \},
		$$
		where $C_{t+1}^{1},..., C_{t+1}^{k_{t+1}}$ are obtained by merging the connected components
		of $H_{t}$ i.e.\ if $D^{i}_{t}\subseteq  \mathcal{C}_{t}$ induces a connected component in $H_{t}$
		then $C^{i}_{t+1} = \bigcup_{C\in D^{i}_{t}} C$.
		
		\item If $|\mathcal{C}_{t+1}|=1$ STOP. Otherwise, go to step 2. 
	\end{enumerate}     
\end{alg}

\begin{defn}\label{percolationdef}
	\begin{itemize}
		\item We say that the $r$-fold graph $\mathbf	{G}=(V, E_{1}, E_{2} ,..., E_{r})$ \textit{percolates} if
		Algorithm~\ref{percolation:definition} applied to $\mathbf{G}$ ends with one single cluster.
		Otherwise we say that $\mathbf{G}$ \textit{does not percolate}.
		\item We say that a subset $W \subseteq V$ is a \textit{percolating subset} (or that it percolates) in
		$\mathbf{G}=(V, E_{1}, E_{2} ,..., E_{r})$ if the \emph{induced $r$-fold subgraph}
		$\mathbf{G}[W]:=(W,E_{1}[W], ..., E_{r}[W])$ percolates.
	\end{itemize}
\end{defn}

The definition of a percolating subset corresponds to the definition in~\cite{jigsawBRSS} of an internally spanned set.

\subsection{Intuition}\label{sec:intuition}

Let us consider heuristically how the jigsaw process might be expected to evolve. For simplicity
we discuss the case $r=2$, although the intuition is transferrable to a larger number of colours.

We begin with $n$ clusters each containing a single vertex. Initially clusters can only merge if there is 
a double-edge (i.e.\ both a red and a blue edge) between the corresponding vertices. Although
such double-edges are rare, the fact that there are many vertices will mean that some clusters
will indeed merge.

Subsequently clusters may continue to merge and grow larger. Indeed, the larger a cluster becomes,
the more likely it is to merge with other clusters and continue growing. Thus we might expect that
after a certain size we encounter a snowball effect, and the growth of the largest cluster accelerates
until it contains all vertices.

Indeed, this intuition turns out to be correct: there is a \emph{bottleneck} in the percolation process,
which occurs at size $\Theta(\ln n)$ (this was observed by Bollob\'as, 
Riordan, Slivken and Smith in \cite{jigsawBRSS}). More precisely, in the subcritical case we show that
the largest cluster in the percolation process will not exceed size $\ln n$ whp. On the other hand,
in the proof of the supercritical process, the hardest part is proving that there is a cluster of size
slightly larger than $\ln n$ - then it is fairly easy to prove that this cluster will eventually merge with
all other clusters whp, and therefore we have percolation.

We will ignore floors and ceilings throughout the paper whenever they do not
significantly affect the arguments (this is usually the case since we consider
graphs on n vertices, where $n \rightarrow \infty$). We also assume that $n$
is sufficiently large in calculations.

\section{Proof of the supercritical case.}
\label{supercritical}

In this section we will prove part~(ii)  of Theorem~\ref{MainResultNonConst}.
The main idea for the proof is to construct an increasing sequence of percolating subsets
$V_{1} \subseteq V_{2} \subseteq V_{3}=V$. Therefore we will divide the proof into three
parts, and we aim to prove the following:

\begin{flushright}
	\begin{itemize}
		\item[] Part I: whp there is a percolating subset $V_{1}\subseteq V$ of size at least
		\textit{$t_{1}$}$:=(\ln n)^{1 + \frac{1}{r}}$;
		
		\item[] Part II: conditioned on the existence of a percolating subset $V_{1}\subseteq V$
		 of size at least $t_{1},$  whp there exists a percolating subset $V_{2} \supset V_{1}$ 
		 of size at least $\frac{n}{2^{r+2}}$;
		
		\item[] Part III: conditioned on the existence of a percolating subset $V_{2}$ of size at
		 least $\frac{n}{2^{r+2}},$ whp the whole set V percolates.
	\end{itemize} 
\end{flushright}

The independence between the three parts
of the proof is guaranteed by independent
rounds of exposure.
More precisely, let 
$\mathbf{G}^{(j)}:= ([n], E^{(j)}_{1}, E^{(j)}_{2},..., E^{(j)}_{r})\sim 
G(n,\frac{p_1}{3},\frac{p_2}{3},\ldots,\frac{p_r}{3})$
independently for $j=1,2,3$. Then we will view $\mathbf{G}$ as the union 
$\mathbf{G}^{(1)}\cup \mathbf{G}^{(2)} \cup \mathbf{G}^{(3)}$.
\footnote{Note that this is not quite true, since the union of three independent 
copies of $G(n,p/3)$ is distributed as $G(n,p^*)$, where $p^*=p-p^2/3+p^3/9$. 
However, since $p^*<p$ we can couple $G(n, p^\ast)$ with $G(n, p)$ such that $G(n, p^\ast) \subseteq G(n, p)$, 
and since percolation is a monotone increasing property, this will be sufficient.}

In Part~$j$ of the proof we will work only with $\mathbf{G}^{(j)}$, effectively \emph{exposing} 
an $r$-fold probability of \linebreak[4]
 $(p_1/3,p_2/3,\ldots,p_r/3)$ in each round.

\subsection{Preliminaries}

We begin with some basic observations.

\begin{prop}\label{p2p}
	Let $r, C_{r}, p_{1} ,p_{2}, ... , p_{r}$ satisfy the conditions of Theorem~\ref{MainResultNonConst}~$(ii)$. Then for $n$ large enough there exist real numbers $0 \leq p_{1}'\leq p_{2}' \leq ... \leq p_{r}' \leq 1$ that also satisfy conditions of Theorem~\ref{MainResultNonConst}~$(ii)$ and such that
	\begin{itemize}
		\item $p_{i}' \leq p_i$ for every $i,$
		\item $p_{1}' p_{2}' ... p_{r}' = \frac{C_{r}}{n (\ln n)^{r-1}}.$
	\end{itemize}
\end{prop}

\noindent{We omit the proof of this intuitively obvious result -- for details see \cite{Thesis}.}

Since percolation is a monotone property, by Proposition~\ref{p2p} we may assume that
\begin{equation}\label{eq:productbound}
P_{r}= p_{1}...p_{r}= \frac{C_{r}}{n (\ln n)^{r-1}}.
\end{equation}
From this, and recalling that $p_{1} \geq \frac{C_{r}\ln n}{n}$, we can deduce that
\begin{equation}\label{eq:boundp2}
p_{2} \leq \left(\frac{p_{1}p_{2}p_{3}...p_{r}}{p_{1}}\right)^{1/(r-1)} \leq \left( \frac{1}{\ln n} \right)^{\frac{r}{r-1}}.
\end{equation}

\begin{remark}
	In the two-coloured case, i.e.\ $r=2$, we obtain the bound $p_1\le p_2 \le (\ln n)^{-2}$. 
	In the general case, the analogous calculation only yields 
	the bound $p_i \le (\ln n)^{-1}$ (for $i \geq 3$). This seemingly minor
	difference leads to significant extra difficulties, as some approximations are
	no longer valid. We will therefore have to distinguish between ``small'' and ``large'' $p_i$ 
	(see Lemmas~\ref{lemma:azul} and~\ref{lemma:partido} in Section~\ref{part1}).
\end{remark}

\subsection{Part I}\label{part1}

We will construct a large percolating subset $V_{1}$ by ``trial and error''. 
Algorithm~\ref{1-1algorithm} will start from a single vertex and add one vertex 
at a time in an attempt to construct $V_{1}$. We will make several attempts to construct $V_{1}$ -- each such attempt is called a \textit{round}; each round consists of a number of \textit{steps}. We divide the proof into two stages:

\begin{itemize}
	\item[I.a] First, we will bound from below the probability that the algorithm 
	constructs a percolating subset of size at least $t_{0}:= \frac{\ln n}{c_{r}}$ 
	(in one round, see Lemma~\ref{lemma:tiempo0}) where $c_{r} := C^{\frac{1}{r-1}}_{r}$. 
	\item[I.b] Second, conditioned on the algorithm constructing a percolating subset 
	of size at least $t_{0},$ we will bound from below the probability that the algorithm 
	constructs a percolating subset of size at least $t_{1}= (\ln n)^{1+\frac{1}{r}}$ 
	(in one round, see Lemma~\ref{lemma:tiempo1}). 
\end{itemize}

The probability that Algorithm~\ref{1-1algorithm} reaches $t_{1}$ in one round
is bounded from below by the product of the probabilities of the two stages.
This product turns out to be small, but crucially
Algorithm~\ref{1-1algorithm} makes many attempts to reach $t_{1}$.
The probability that at least one of these rounds succeeds will be large
(see Lemma~\ref{main}).

In step $t$ of round $k$ of Algorithm~\ref{1-1algorithm}, we have a trial set $X^{t}_{k}$ which is a percolating set. If the algorithm 
finds a suitable vertex to add to the trial set $X^{t}_{k}$, we create the new trial set $X^{t+1}_{k}$ 
and proceed to step $t+1$ of round $k$. If not, we discard the vertices of the trial set $X^{t}_{k}$ 
and begin the new round $k+1$. We stop if a round has reached step $t_{1}$ or if we have had $\frac{n}{2t_1}$ rounds.

The formal description of the algorithm is as follows:

\begin{alg}	[The 1-by-1 algorithm] \label{1-1algorithm}
	The algorithm is divided into rounds, indexed by k, and each round is divided into steps, indexed by t. At the start of the $k$-th round there is a set $A^{0}_{k} \subseteq [n]$ of active vertices and a set $D_{k} \subseteq [n]$ of discarded vertices. We begin with $A^{0}_{1} = [n]$ and $D_{1}= \varnothing.$ The procedure of the $k$-th round is as follows:
	\vspace{.3cm}
	
	\noindent{At the start of the $t$-th step of the k-th round there are sets of trial and dormant vertices:}
	\begin{center}
		\hspace{.5cm} $\bullet$ $X_{k}^{t} = \{ x_{k}^{1},x_{k}^{2}, ... ,x_{k}^{t}\}\subseteq A_{k}^0 $ (trial vertices); \hspace{1cm}
		$\bullet$ $U_{k}^{t} \subseteq A_{k}^{0}$ (dormant vertices),
	\end{center}
	
	\noindent{where $A_{k}^{0} = X_{k}^{t} \, \dot{\cup} \, A_{k}^{t} \, \dot{\cup} \, U_{k}^{t}.$ }

		\item[(1)] For $t = 0,$ we move an arbitrary active vertex $x_{k}^{1} \in A_{k}^{0}$ to the trial set:
		\begin{center}
			$\bullet$ $ X_{k}^{1} := \{x_{k}^{1}\}$; \hspace{1cm}
			$\bullet$ $ U_{k}^{1} := \varnothing$; \hspace{1cm}
			$\bullet$ $ A_{k}^{1} := A_{k}^{0} \backslash {x^{1}_{k}}$; \hspace{1cm}
			$\bullet$ $ R_{k}^{0} := \varnothing$,
		\end{center}
	    and set $t := 1.$
		
		\item[(2)] For $t \geq 1,$ we reveal all edges of $E_{1}^{(1)}$ between $A_{k}^{t}$ and $x_{k}^{t}$ and edges of $E_{i}^{1}$ ($i= 1,\ldots,r$) between any neighbour of $x_{k}^{t}$ in $E_{1}^{(1)}$ and $x_{k}^{1}, \ldots,x_{k}^{t}$. Let\\
				$\bullet$ $R_{k}^{t} := \{x \in A_{k}^{t}: xx_{k}^{t} \in E_{1}^{(1)}\}$;\\
				$\bullet$ $B_{k}^{t} := \{ x \in R_{k}^{t} : \textrm{for every $i \in \{2,3,...,r\}$ there exists $s_{i} \leq t$ such that } xx_{k}^{s_{i}} \in E_{i}^{(1)} \}$.

			\item[(3)] If $B_{k}^{t} \neq \varnothing, $ then let $x_{k}^{t+1}$ be an arbitrary element of $B_{k}^{t}.$ Then set:
			\begin{center}
				$\bullet$ $X_{k}^{t+1}:= X_{k}^{t} \cup \{ x^{t+1}_{k} \}$; \hspace{1cm}
				$\bullet$ $A_{k}^{t+1} := A_{k}^{t} \backslash R_{k}^{t}$;  \hspace{1cm}
				$\bullet$ $U_{k}^{t+1} := U_{k}^{t} \cup \left( R_{k}^{t} \backslash \{ x_{k}^{t+1} \} \right)$.
			\end{center}
		 If $t \geq t_{1} = (\ln n)^{1+ \frac{1}{r}}$ then STOP, otherwise set $t=t+1$ and go to step (3).
		
		\item[(4)] If $B_{k}^{t}= \varnothing,$ then set
		\begin{center}
			$\bullet$ $A_{k+1}^0:= A_{k}^0 \backslash X_{k}$; \hspace{.5cm}
			$\bullet$ $D_{k+1}:=D_{k} \cup X^{t}_{k}$.
		\end{center}
		\item[(5)] If		
		\begin{equation*}
		k \geq \frac{n}{2(\ln n)^{1 + \frac{1}{r}}} 
		\end{equation*}
		
		then STOP, otherwise set $k:= k+1$ and $t:=0,$ and go to step (1).

\end{alg}

We reveal edges and non-edges as they are exposed in the algorithm, e.g.\ when defining 
$R_{t}^{k}$ we \textit{test} each pair $(x_{0}^{k}, a)$ for $a \in A_{k}^{t}$ to reveal 
whether it lies in $E_{1}^{(1)}$. Note that since every tested pair has at least 
one of its endpoints in the trial set, we guarantee independence between rounds by 
discarding the trial set at the end of each round. We also have independence within 
each round, because  no pair is tested twice within a round.

	Since we consider at most $n / (2 (\ln n)^{1+\frac{1}{r}})$ rounds, and stop each with a trial set of size at most $(\ln n)^{1 + \frac{1}{r}} $ vertices,  we start each new round with at least $n/2$ vertices, i.e.

	\begin{equation*}
	| A_{k}^{0} | \geq \frac{n}{2}.
	\end{equation*}

We will need the following definitions:

\begin{defn}
	\begin{itemize}
		\item Let $\mathcal{X}_{k}^{t}$ be the event that $X_{k}^{t}$ is defined (i.e.\ we reach step $t$ in round $k$). 
		\item Let $\mathcal{S}_{k}^{t}:=\{ |R_{k}^{s}| \leq \frac{n}{4  t_{1}} \textrm{ for } s = 0,1,2,...,t \}.$ 
		\item Let $ \Y{k}{t} := \mathcal{X}_{k}^{t}\cap\mathcal{S}_{k}^{t}.$
		\item Let $r_{k}^{t}:= \mathbb{P}\left[ \Y{k}{t} \big| \Y{k}{t-1} \right]$ for $k \leq n/( 2(\ln n)^{1 +\frac{1}{r} })$ and $t \geq 1.$ 
	\end{itemize}
\end{defn}

		The event $\mathcal{X}_{k}^{t}$ 	means that we found a percolating subset of size $t$ formed with only edges of the first 				round of exposure. Conditioned on getting to round $k$ the event $\mathcal{X}_{k}^{1}$ always holds. For $t \geq 2$ 		the event $					\mathcal{X}_{k}^{t}$ is equivalent to the event that $B_{k}^{t-1}$ is non-empty. 
		The event $\mathcal{S}_{k}^{t}$ guarantees that within a round k, we do not discard too many vertices by step $t$. More 					specifically, if the event $	\Y{k}{t-1}$ holds, we have
		\begin{equation*}	
		|A_{k}^{t}| \geq  |A_{k}^{0}| - (t-1)\frac{n}{4t_1}  \geq \frac{n}	{2} - \frac{t}{t_{1}} \frac{n}{4} \geq \frac{n}{4}.
		\end{equation*}
		Note that if we get to round k, the event $\mathcal{S}_{k}^{0}$ always holds,
		since $R_k^0=\emptyset$.

We will use the following easily verified inequalities to approximate some expressions.

\begin{fact}\label{claim:pt-bounds}
	\label{talacha}
	For $t \geq 0,$ $ p \leq 1 $ we have
	\begin{itemize}
		\item[a)] If $ 1-pt \geq 0$ then $1 - (1-p)^{t} \geq pt(1-pt);$ 
		\item[b)] If $ 1-pt \leq \frac{1}{2}$  then  $ 1-(1-p)^{t} \geq \frac{1}{5}.$
	\end{itemize}
\end{fact}

Note that a) was used in \cite{jigsawBRSS}, but that b) is only needed for the multi-coloured case. 
We will also use the following observation: for $t \leq t_{1}$,
\begin{equation}\label{eq:p2t}
	p_{1} t \leq p_{2} t_{1} \stackrel{\eqref{eq:boundp2}}{\leq} \frac{(\ln n)^{1+ \frac{1}{r}}}{ (\ln n)^{\frac{r}{r-1}}} = (\ln n)^{ \frac{1}{r} - \frac{1}{r-1}}  = o(1).
\end{equation}

The following parameter will help us distinguish between ``small'' and ``large'' $p_{i}$'s, 
something that is not needed in the 2-coloured case since both $p_{1}$ and $ p_{2}$ are ``small''.

\begin{defn}\label{i_t}
	For $t \leq t_{1}$, let $i_{t}:= \max\{ i \in [2,r] : 1-\frac{p_{i}t}{3} \geq \frac{1}{2} \}$. 
\end{defn}

Note that by (\ref{eq:p2t}), $i_{t}$ is well defined.

We now calculate a lower bound on the probability of ``one-step success''
i.e.\ the probability of being able to add a vertex to the percolating set
in Algorithm~\ref{1-1algorithm}. Recall that $P_{i}= p_{1}p_{2}...p_{i}$
for each $ 1 \leq i \leq r $.

\begin{lemma} \label{lemma:azul}
	For $n$ large enough and $1 \leq t \leq t_{1} = (\ln n)^{1+\frac{1}{r}}$ 
	we have that independently for each $x \in \mathcal{A}_{k}^{t}$ the following holds:
	\begin{equation*}
	\mathbb{P}[x \in B_{k}^{t}] 
	\geq  \left( \frac{1}{5} \right)^{r-1} \frac{P_{i_{t}}}{3^{i_t}} t^{i_{t}-1}.
	\end{equation*}

\end{lemma}

\begin{proof}[\textbf{Proof}] We have
	\begin{align*}
	\mathbb{P}[x \in B_{k}^{t}] = \frac{p_{1}}{3} \prod_{j=2}^{r} \left(1-\left(1-\frac{p_{j}}{3}\right)^{t}\right)
	&\stackrel{\mbox{\footnotesize (F.~\ref{claim:pt-bounds})}}{\geq }
	\frac{p_{1}}{3} \left(\frac{1}{5}\right)^{r-i_{t}} \prod_{j=2}^{i_{t}}\frac{p_{j}t}{3}\left(1-\frac{p_{j}t}{3}\right)\\
	&\geq 
	\left( \frac{1}{5} \right)^{r-i_{t}} \frac{P_{i_{t}}}{3^{i_t}} t^{i_{t}-1} \left(\frac{1}{2}\right)^{i_{t}-1}
	\geq 
	\left( \frac{1}{5} \right)^{r-1} \frac{P_{i_{t}}}{3^{i_t}} t^{i_{t}-1}.\qedhere
	\end{align*} 
\end{proof}

We now make use of the lower bound of Lemma~\ref{lemma:azul} and the fact that
the events $\{ x \in B_{k}^{t} \}$ are independent
for different vertices $x$. Recall that $c_r:=C_r^{\frac{1}{r-1}}$.

\begin{lemma}
	\label{lemma:partido}
	
	For  $ 1 \leq t \leq t_{1}= (\ln n)^{1+\frac{1}{r}}$ and $n$ large enough the following holds
	\begin{itemize}
		\item[(a)] $\mathbb{P}\left[ \Y{k}{t} \big|\Y{k}{t-1} \right]
		\geq 1 - \exp \left( - \left(\frac{1}{16}\right)^{r}  \left(\frac{c_{r}t}{\ln n}\right)^{i_{t}-1}\right).$
		\item[(b)] If $\left(\frac{1}{16}\right)^{r} \left(\frac{c_{r}t}{\ln n}\right)^{i_{t}-1} \leq 1$ we have
		$$
		\mathbb{P}\left[ \Y{k}{t} \big|\Y{k}{t-1} \right] \geq \left(\frac{1}{32}\right)^{r} \left(\frac{c_{r}t}{\ln n}\right)^{i_{t}-1}.
		$$
	\end{itemize}
\end{lemma}

\begin{proof}[\textbf{Proof}]
	Since case (b) follows from case (a) and the inequality $1 - \frac{x}{2} \geq \exp( -x) $, 
	valid for $x \in [0,1],$ we only need to prove case (a).
	
	We recall that $\mathcal{Y}^{t}_{k} = \mathcal{X}^{t}_{k} \cap \mathcal{S}^{t}_{k}$, and therefore
	
	\begin{equation}\label{eq:LowerBoundY}
	\mathbb{P}\left[  \Y{k}{t} \big| \Y{k}{t-1} \right]
	\geq 
	1 - \mathbb{P}\left[ \overline{\mathcal{X}_{k}^{t}} \big| \Y{k}{t-1} \right]
	- \mathbb{P}\left[ \overline{\mathcal{S}_{k}^{t}} \big| \Y{k}{t-1} \right].
	\end{equation}
	We bound the two probability terms on the right-hand side of the inequality separately.
	
	Let $\mathcal{Z}_{k}^{t}$ be the random variable that represents the number of sets $Z$ of size 
	$ \frac{n}{4t_{1}} $ such that $Z \subseteq R^{t}_{k}$. If $\bar{S}^{t}_{k} \cap \Y{k}{t-1}$ holds 
	then $\mathcal{Z}^{t}_{k} \geq 1$, therefore we deduce using Markov's inequality that
	\begin{equation}\label{eq:LowerBoundS}
	\mathbb{P}\left[  \overline{\mathcal{S}_{k}^{t}} \big| \Y{k}{t-1} \right]
	\leq \mathbb{E}\left[ \mathcal{Z}_{k}^{t}  \Big| \Y{k}{t-1} \right] 
	\leq {\binom{n}{\frac{n}{4 t_{1}}}} \left(\frac{p_1}{3}\right)^{\frac{n}{4 t_{1}}}  
	\leq 
	\left( \frac{4}{3} e t_{1} p_{1} \right)^{\frac{n}{4 t_{1}}} 
	\leq e^{-\sqrt{n}}.
	\end{equation}
	For the last inequality we used that $n/(4t_{1}) \geq \sqrt{n}$ and
	$p_{1}t_{1} \stackrel{\eqref{eq:p2t}}{=} o(1) \leq \frac{3}{4e^2}$ for $n$ large enough.

	For the second term in \eqref{eq:LowerBoundY}, we use Lemma~\ref{lemma:azul} and the observation that 
	$|A_{k}^{t-1}|\geq n/4$ to obtain
	\begin{align*}
	\mathbb{P}\left[ \overline{\mathcal{X}_{k}^{t}} \big| \Y{k}{t-1} \right] 
	= \prod_{x \in A_{k}^{t-1}} \mathbb{P}[x \notin B_{k}^{t-1}]
	\leq \left( 1 -  \left( \frac{1}{5} \right)^{r-1}  \frac{P_{i_{t}}}{3^{i_t}} t^{i_{t}-1} \right)^{\frac{n}{4}}
	\leq \exp\left(-\frac{n}{4} \left( \frac{1}{5} \right)^{r-1} \frac{P_{i_{t}}}{3^{i_t}} t^{i_{t}-1}\right).
	\end{align*}
	From the assumptions of Theorem~\ref{MainResultConst}, we have that 
	$P_{i_{t}} \geq \frac{ C_{r} }{ n (\ln n)^{i_{t}-1} }\geq\frac{c^{i_{t}-1}_{r}}{n (\ln n)^{i_{t}-1}}.$ 
	We deduce that

	\begin{equation}\label{eq:LowerBoundX}
		\mathbb{P}\left[ \overline{\mathcal{X}_{k}^{t}} \big | \Y{k}{t-1} \right]
		\leq \exp\left(- \left( \frac{1}{15} \right)^{r}  \left(  \frac{c_{r}t}{\ln n}\right)^{i_{t}-1}\right).
	\end{equation}
	
	\noindent{Substituting \eqref{eq:LowerBoundS} and \eqref{eq:LowerBoundX} into \eqref{eq:LowerBoundY} gives}
	$$
	\mathbb{P}\left[  \Y{k}{t}\big | \Y{k}{t-1} \right] 
	\geq 1-   \exp\left(-\left( \frac{1}{15} \right)^{r}  \left( \frac{c_{r}t}{\ln n}\right)^{i_{t} -1}\right) - \exp(-\sqrt{n}).
	$$
	
	To complete the proof we recall that $t\leq t_{1},$ $ 2 \leq i_{t}\leq r$ and observe that
	$$
		\left(\frac{t}{\ln n}\right)^{i_{t}-1}
		\leq
		\left( \frac{t_{1}}{\ln n} \right)^{r-1}
		=
		(\ln n)^{\frac{r-1}{r}}=o(\sqrt{n}),
	$$
	and conclude that
	\begin{align}
		\mathbb{P}\left[  \Y{k}{t} \big| \Y{k}{t-1} \right] 
		\geq 1-   \exp\left(-\left(\frac{1}{16}\right)^{r}  \left( \frac{c_{r}t}{\ln n}\right)^{i_{t}-1}\right).
		&\qedhere
	\end{align}

\end{proof}

Recall that $t_{0}= \frac{\ln n}{C^{1/(r-1)}_{r}} =\frac{\ln n}{c_{r}}$. In order to calculate 
a lower bound on the  probability of ``proceeding to step $t_{1}$'' we use Lemma~\ref{lemma:partido} 
to calculate lower bounds for the events ``proceeding to step $t_{0}$'' and ``proceeding to step $t_{1}$ 
given that we already proceeded to step $t_{0}$''. We formally express this in Lemmas~\ref{lemma:tiempo0} 
and \ref{lemma:tiempo1}.

\begin{lemma}\label{lemma:tiempo0}
	$\mathbb{P}\left[\Y{k}{t_{0}} \big| \mathcal{X}_{k}^{0} \right] \geq n^{-\frac{7(r-1)}{c_{r}}}.$ 
\end{lemma}

\begin{proof}[\textbf{Proof}]
	Since $\frac{c_{r}t}{\ln n} \leq 1$ for $1 \leq t \leq t_{0}$, we can use Lemma~\ref{lemma:partido}~(b):
	\begin{align*}
	\mathbb{P}\left[ \Y{k}{t_{0}} \big| \mathcal{X}_{k}^{0}  \right]
	=
	\mathbb{P}\left[ \Y{k}{t_{0}} \big| \mathcal{Y}_{k}^{0}  \right]
	=
	\prod_{t = 1}^{t_{0}} \mathbb{P}\left[ \Y{k}{t} \big| \Y{k}{t-1} \right]	
	\geq 
	\prod_{t=1}^{t_{0}} \left(\frac{1}{32}\right)^{r} \left( \frac{C_{r} t}{\ln n} \right)^{i_{t}-1}
	&\geq
	\prod^{t_{0}}_{t=1} \left(\frac{1}{32^{2}}\right)^{r-1} \left( \frac{c_{r} t}{\ln n} \right)^{r-1}\\
	&\geq
	\left( \frac{c_{r} t_{0}}{1024 \ln n} \right)^{(r-1)t_{0}}\\ 
	&= 	
	\left( \frac{1}{1024} \right)^{(r-1)\frac{\ln n}{c_{r}}} 
	\geq
	n^{-\frac{7(r-1)}{c_{r}}},
	\end{align*}
	since $\frac{1}{1024}\geq \frac{1}{e^{7}}$.
\end{proof}

\begin{lemma}\label{lemma:tiempo1}
	$\mathbb{P}\left[\mathcal{X}_{k}^{t_{1}}\big|\Y{k}{t_{0}}\right]\geq n^{-2^{8r+2}/c_{r}}.$
\end{lemma}

In the proof of Lemma~\ref{lemma:tiempo1} we will use the following claim.

\begin{claim}\label{claim:alphabound}
For any real numbers $\alpha\ge 1$ and $0 \le y \le 1-\frac{1}{\alpha}$, we have
$$
1-y \ge e^{-\alpha y}.
$$
\end{claim}

\begin{proof}
	From the hypothesis we deduce that
	\[
		\alpha \geq \frac{1}{1-y} 
		= 
		\sum_{i = 0}^{\infty} y^i 
		\stackrel{(y \geq 0)}{\geq} 
		\frac{1}{y} \sum_{i=1}^{\infty} \frac{y^i}{i} 
		= 
		\frac{-\ln(1-y)}{y}, 
	\]
    and the desired inequality follows.
\end{proof}

\begin{proof}[\textbf{Proof of Lemma~\ref{lemma:tiempo1}}]
	We begin by applying Lemma~\ref{lemma:partido}~(a):
	\begin{align*}
	\mathbb{P}\left[  \mathcal{X}_{k}^{t_{1}}  \big| \Y{k}{t_{0}}    \right] 
	\geq
	\mathbb{P}\left[  \mathcal{Y}_{k}^{t_{1}}  \big| \Y{k}{t_{0}}    \right] 
	=
	\prod_{t = t_{0}+1}^{t_{1}}\mathbb{P}\left[\Y{k}{t}\big|\Y{k}{t-1}\right]
	\geq
	\prod_{t = t_{0}}^{t_{1}} \left( 1 - \exp \left( - \left(\frac{1}{16}\right)^{r}  
	\left(\frac{c_{r}t}{\ln n}\right)^{i_{t}-1} \right) \right). 
	\end{align*}
	
	Setting $\alpha := \frac{1}{1- \exp \left\{ -\left( \frac{1}{16} \right)^{r} \right\} } > 1$ and
	 $y :=  \exp\left( -\left(\frac{1}{16}\right)^{r} \left(\frac{c_{r}t}{\ln n} \right)^{i_{t}-1} \right)$ 
	 and noting that $\frac{c_{r} t}{\ln n} \geq  \frac{c_{r} t_{0}}{\ln n}  =  1$
	 for $t \geq t_{0}$, we deduce that $y \leq \exp\left(- \left(\frac{1}{16}\right)^{r} \right) = 1 - \frac{1}{\alpha} $,
	 therefore we can apply Claim~\ref{claim:alphabound}. Thus
	\begin{align*}
	\mathbb{P}\left[  \mathcal{X}_{k}^{t_{1}}  \big| \Y{k}{t_{0}}    \right] 
	\geq
	\exp\left( -\alpha \sum_{t = t_{0}}^{t_{1}} 
	\exp \left( - \left(\frac{1}{16}\right)^{r} \left(\frac{c_{r} t}{\ln n}\right)^{i_{t}-1} \right)  \right)  
	&\geq \exp\left( -\alpha \sum_{t = t_{0}}^{\infty} 
	\exp \left( - \left(\frac{1}{16}\right)^{r}   \left(\frac{c_{r} t}{\ln n}\right) \right)
	\right)\\
	&= \exp \left( -\frac{\alpha \exp \left(- \left(\frac{1}{16}\right)^{r}  \right) }{ 1 - \exp\left( -\left( \frac{1}{16} \right)^{r} \frac{c_{r}}{\ln n} \right) } \right).
	\end{align*}
	
	We now simplify the denominator by using the inequality $ e^{-x} \leq 1- x/2$ valid for $x\leq 1:$
	\begin{align*}
	\mathbb{P}\left[  \mathcal{X}_{k}^{t_{1}}  \big | \Y{k}{t_{0}}    \right] 
	\geq \exp \left( -\frac{\alpha \exp \left(- \left(\frac{1}{16}\right)^{r}  \right) }{ 1 - \left( 1-\frac{1}{2}\left( \frac{1}{16} \right)^{r} \frac{c_{r}}{\ln n}  \right) } \right)
	=
	\exp\left( -\frac{ 2^{4r+1} \alpha  \exp \left( -\left(\frac{1}{16}\right)^{r} \right) \ln n }{c_{r}} \right).
	\end{align*}
We now observe that
\begin{align*}
2^{4r+1}\alpha \exp\left(-\left(\frac{1}{16}\right)^r\right) = 2^{4r+1}\frac{\exp\left(-\left(\frac{1}{16}\right)^r\right)}{1-\exp\left(-\left(\frac{1}{16}\right)^r\right)}
\le 2^{4r+1} \frac{1}{\frac{1}{2}\left(\frac{1}{16}\right)^r}
= 2^{8r+2},
\end{align*}
and the result follows.
\end{proof}

Using Lemmas~\ref{lemma:tiempo0} and \ref{lemma:tiempo1}, we can complete the proof of Part I

\begin{lemma}\label{main}
	$\mathbf{G}^{(1)}$ contains a percolating subset of size $(\ln n)^{1+ \frac{1}{r}}$ 
	with probability at least $1 - e^{-\sqrt{n}}$.
\end{lemma}

\begin{proof}[\textbf{Proof}]	
	
	Let $k \leq \frac{n}{2(\ln n)^{1+ \frac{1}{r}}}.$ Applying Lemmas~\ref{lemma:tiempo0}~and~\ref{lemma:tiempo1}, 
	the probability that in round $k$ we find a percolating subset of size $(\ln n)^{1 + \frac{1}{r}}$ is at least
	\[
	n^{-\frac{7(r-1)}{c_r}} \cdot n^{-\frac{2^{8r+2}}{c_r}} \ge n^{-\frac{2^{8r+3}}{c_r}}.
	\]
	We conclude that the probability of not finding a percolating subset of size
	$t_{1}= (\ln n)^{1+\frac{1}{r}}$ in each of the $n/\left(2(\ln n)^{1+ \frac{1}{r}}\right)$
	rounds is at most
	$$
	\left( 1- n^{-\frac{2^{8r+3}}{c_r}} \right)^{\frac{n}{2(\ln n)^{1+ \frac{1}{r}}}}
	\leq 
	\exp\left( - \frac{n^{1-\frac{2^{8r+3}}{c_r}}}{2(\ln n)^{1+\frac{1}{r}}}\right)
	\leq 
	\exp(- \sqrt{n}).
	$$
	These inequalities hold since $c_{r} \ge 2^{8r+5}$, provided $n$ is large enough compared to $c_{r}$.
\end{proof}

\begin{remark}
	We note that as $r$ becomes larger, Algorithm~\ref{1-1algorithm} has a harder time constructing a 
	percolating set larger than $\ln n$. While for two colours we reach size $(\ln n)^{\frac{3}{2}}$ whp, 
	for $r$ colours we must settle for size $(\ln n)^{1+\frac{1}{r}}$.
\end{remark}

\subsection{Part II}\label{part2}

In this subsection we aim to prove that conditioned on the existence of a percolating 
set of size $t_{1}$ in $G^{(1)}$, whp there is  a percolating set of size at least
 $\frac{n}{2^{r+2}}$ in $G^{(1)}\cup G^{(2)}$ (see Lemma~\ref{lemma:part2}). 

We will attempt to construct a percolating set of linear size with the following algorithm:

\begin{alg}[The doubling algorithm]\label{doubling}
    \hspace{1cm} \newline Input: an $r$-fold graph $G^{(2)}$ and a subset $X_{0}$ which 
    is percolating with respect to $G^{(1)}$.
    
    For $t\geq 0$, we construct $X_{t}$ inductively as follows:
	\begin{itemize}
		
		\item Let $A_{t}:= V \setminus X_{t} $ be the set of active vertices.
	\end{itemize}
	
	\begin{enumerate}
		\item[(1)] At step $t \geq 0$ we reveal all edges of $G^{(2)}$ between $A_{t}$ and 
		$X_{t}\setminus X_{t-1}$, where $X_{-1}:= \varnothing$.  We define
		
		\begin{itemize}
			\item 
			$
			B_{t}:= \{ v \in A_{t}: \mbox{ $\forall \  i \in [r]$ there is a $v_{i} \in X_{t}\setminus X_{t-1}$
				such that $v v_{i} \in E_{i}^{(2)}\}.$}
			$
		\end{itemize}
		
		In other words, $B_{t}$ is the set of active vertices joined to $X_{t}\setminus X_{t-1}$ 
		by an edge of each colour from the second round of exposure.

		\item[(2)] If $|B_{t}|< |X_{t}|$ we STOP. Otherwise, we set
		
		\begin{itemize}
			\item $X_{t+1} := X_{t} \cup B_{t}$,
			\item $A_{t+1} := A_{t}\setminus B_{t}.$
		\end{itemize}
		
		If $|X_{t+1}| \geq n /2^{r+2} $ then STOP, otherwise go to (1) for step $t+1$.
	\end{enumerate}
\end{alg}

	We set $b_{t}:=|B_{t}|$ and $x_{t}:=|X_{t}|$ for all $t$.

\begin{remark} \label{remark:double}
	\begin{enumerate}[(i)]
		\item If we reach step $t+1$ in Algorithm~\ref{doubling}, then $b_{s}\geq x_{s}$ 
		for every $s \in [t]$ and therefore
		\[
		x_{s}= x_{s-1}+b_{s-1}
		\geq
		2x_{s-1} \mbox{ for every } s\in [t+1].
		\]
		
		Thus $b_{t} \geq x_{t}\geq 2 x_{t-1}\geq 2^{2} x_{t-2}\geq ... \geq 2^{t}x_{0}=2^{t}t_{1}$.
		\item If we reach step $t+1$, then
		\begin{equation}\label{eq:xtbt}
		x_{t+1}= b_{t}+x_{t}\leq 2b_{t}.
		\end{equation}		
	\end{enumerate}
\end{remark}

	Let \textit{$t_{2}$}$:=\max \left\{t \in \mathbb{N}\cup \{0\} : x_{t} < \frac{n}{2^{r+2}}\right\}$. 
	Note that if Algorithm~\ref{doubling} constructs a percolating set $X_{t}$ of
	 size $\geq \frac{n}{2^{r+2}}$, then it will stop at time $t = t_{2}+1$; 
	 otherwise it will stop at time $t_{2}$. 
	Furthermore, by the previous remark we know that $2^{t_{2}} \leq \frac{b_{t_{2}}}{t_{1}}\le n $,
	 so $t_{2} \leq \log_{2}(n) =O(\ln n)$.

Given an $r$-fold graph $\mathbf{G}$, we denote the event that $V$ contains a percolating subset 
of size at least m by \textit{$\mathcal{E}(\mathbf{G},m)$}. 
The general idea to prove the main result of this section (Lemma~\ref{lemma:part2}) is as follows: we first prove
in Claim~\ref{auxclaim2} that the expected number of ``suitable'' vertices $B_{t}$ is at least twice the size of the
percolating set $X_{t}$ constructed in step $t-1$ (see Steps~1\&2 of Algorithm~\ref{doubling}). 
Subsequently, in Lemma~\ref{lemma:part2aux} we prove a lower bound
on the conditional probability that Algorithm~\ref{doubling} 
proceeds  to step $t+1$ conditioned on it reaching step $t$. 
Finally, we apply this lower bound multiple times to obtain Lemma~\ref{lemma:part2}.

\begin{claim}\label{auxclaim2}
	Let $t \leq t_{2}.$  Then
	\displaymath 
	\mathbb{E}[b_{t}]  \geq 2 x_{t}.
	\enddisplaymath
\end{claim}

\begin{proof}[\textbf{Proof}]
	
	Let $q_{t,i}$ denote the probability that a vertex $v \in A_{t}$ is joined to
	$B_{t-1}=X_{t}\setminus X_{t-1}$ by at least one edge of
	$G_{i}^{(2)}.$ From~\eqref{eq:xtbt} we know that $b_{t-1} \geq x_{t}/2$ for
	$0\leq t \leq t_{2}$, where $b_{-1}:= x_{0}$, and so we obtain
	
	\begin{equation}\label{equ:laQuebradora}
	q_{t,i} = 1 - \left(1-\frac{p_{i}}{3}\right)^{b_{t-1}} \geq 1- \left(1-\frac{p_{i}}{3}\right)^{x_{t}/2} \geq 1 - \exp\left( -\frac{p_{i}x_{t}}{6} \right)
	\geq
	\begin{cases}
	\frac{p_ix_t}{12} & \mbox{if }p_ix_t\le 6; \\
	\frac12 & \mbox{otherwise.}
	\end{cases}
	\end{equation}
	Let $j_{t}:= \max \{ j \in [r] \cup \{ 0 \} : p_{j} x_{t} \leq 6  \} \geq 0,$ where $p_{0} :=0.$  
	Recalling that $A_{t}\geq n/2$ for $t \leq t_{2},$ we obtain
	
	\begin{equation*}	
	\mathbb{E}[b_t]
	=
	|A_{t}| \left( \prod_{j=1}^{r} q_{t,j} \right)
	\stackrel{(\ref{equ:laQuebradora})}{\ge} 
	\frac{n}{2} \left( \prod_{j=1}^{j_{t}} \frac{p_{j} x_{t}}{12} \right)
	\left( \frac{1}{2} \right)^{r-j_{t}}
	=
	\begin{cases}
	2\left(\frac{n}{2^{r+2}}\right)\ge 2x_t & \mbox{for }j_t=0;\\
	n \left(\frac{x_{t}}{3}\right)^{j_{t}} P_{j_{t}} \left( \frac{1}{2} \right)^{r+j_{t}+1} & \mbox{otherwise.}
	\end{cases}
	\end{equation*}
	Thus we may assume that $j_t\ge 1$ (otherwise we are done).
	Making a further case distinction we obtain \newline
	\textbf{Case 1: $j_{t}=1.$} We recall that $P_{1}=p_{1}\geq \frac{C_{r}\ln n}{n}$, thus 
	for $n$ large enough we have:
	\displaymath
	\mathbb{E}[b_t]
	\geq
	\left(\frac{C_{r} \ln n}{2^{r+2}}\right) \frac{x_{t}}{3}
	\geq
	2x_{t},
	\enddisplaymath
	since $\frac{C_{r}}{2^{r+2}} \geq 1$.
	
	\noindent\textbf{Case 2: $j_{t}\geq 2.$} We recall that $P_{i}= p_{1}...p_{i} \geq C_{r} / (n (\ln n)^{i-1})$ 
	for all $2 \leq i \in [r]$ and $x_{t}\geq 2^{t}t_{1}$ for all $0\leq t\leq t_{2}$. Thus
	\begin{align*}
	\mathbb{E}[b_t] 
	\geq
	\frac{n}{2^{r+j_{t}+1}} \frac{(2^{t}t_{1})^{j_{t}-1}}{3^{j_t}} x_{t} \left( \frac{C_r}{n (\ln n)^{j_{t}-1}} \right)
	= 
	C_{r} \frac{2^{t(j_t-1)}}{2^{r+j_{t}+1} 3^{j_t}} \left(\frac{t_{1}}{\ln n}\right)^{j_{t}-1} x_{t}
	\geq
	\frac{C_{r}}{2^{2r+1}3^r} x_{t}
	\geq
	2x_{t},
	\end{align*}
	where the last two inequalities are valid since $\frac{t_1}{\ln n}\ge 1$ and $C_{r}\ge 2^{8r^2}\ge 2^{2r+2}3^r$.
\end{proof}

  We apply Claim~\ref{auxclaim2} to bound the probability that we are able to
  double the size of the percolating set in each step.
  
\begin{lemma}\label{lemma:part2aux}
	For each integer $1\le t \le t_2$, we have	
	\displaymath		
	\mathbb{P}\left[ b_{t} \geq x_{t} | X_{t} \neq \varnothing \right] 
	\geq
	1- \exp \left( - \frac{t_{1}}{4} \right).
	\enddisplaymath
\end{lemma}

\begin{proof}[\textbf{Proof}]
	
	For $t \leq t_{2},$ the trial set $X_{t}$ is of size at most $n/2^{r+2}$. This means that there 
	are at least $n - n/2^{r+2} \geq n/2$ vertices in the set of active vertices $A_{t}.$ 
	
	We note that the events that $v \in B_{t}$ are independent for different $v\in A_{t}$,
	so $b_t$ is distributed as
	$\mbox{Bi}(|A_{t}|, q_{t,1}q_{t,2}...q_{t,r}),$.
	Note that the distribution of $b_t$ is dependent on both $|A_t|$ and $b_{t-1}=|X_t\setminus X_{t-1}|$.
	In what follows we will suppress the conditioning on these two variables for ease of notation.
	
	Now the
	Chernoff bound (see e.g.\ \cite{ProbMethod1}) tells us that 
		\begin{equation}\label{eq:chernoff}
		\mathbb{P}[\mbox{Bi}(m,q) \leq (1-\delta)mq] \leq \exp\left( - \frac{mq \delta^{2}}{2} \right)\ \mbox{ for all }\ 0 < \delta < 1. 
		\end{equation}
	From Claim~\ref{auxclaim2} we deduce that
	\displaymath 
	\mathbb{P}\left[ b_{t}\geq x_{t} | X_{t} \neq \varnothing \right]
	\geq 
	\mathbb{P}\left[ b_{t} >  \frac{\mathbb{E}[b_{t}]}{2} \right]
	\stackrel{\mbox{\footnotesize \eqref{eq:chernoff}}}{\geq }
	1- \exp ( -\mathbb{E}[b_{t}]/8 )
	\geq 
	1 - \exp ( - x_{t}/4 )
	\geq 
	1 - \exp ( -x_{0}/4  ).
	\enddisplaymath
	Recalling that $x_{0}=t_{1}$, this completes the proof.
\end{proof}

We apply Lemma~\ref{lemma:part2aux} multiple times to obtain the main result of this section.

\begin{lemma}\label{lemma:part2}
	For $n$ large enough,	
	\displaymath
	\mathbb{P}\left[ \mathcal{E} \left(G^{(1)} \cup G^{(2)}, \frac{n}{2^{r+2}}\right) \Big |  \mathcal{E}\left(G^{(1)}, (\ln n)^{1+\frac{1}{r}}\right)  \right]
	\geq
	1 - \exp( -\frac{ t_{1} }{ 5 }  ).
	\enddisplaymath
\end{lemma}

\begin{proof}[\textbf{Proof}]
	
	Since $t_{2} \leq K \ln n$ for some $K$, we deduce from Lemma~\ref{lemma:part2aux} that
	\begin{align*}
	\mathbb{P}\left[ \mathcal{E} \left(G^{(1)} \cup G^{(2)}, n/2^{r+2}\right) \Big |  \mathcal{E}\left(G^{(1)}, t_1\right) \right]
	\geq 
	\prod_{t=0}^{t_{2}}\mathbb{P}\left[ b_{t} \geq x_{t} \big | X_{t} \neq \varnothing \right]
	&\geq
	\left(  1 - \exp\left(  - \frac{t_{1}}{4}  \right)  \right)^{K \ln n}\\
	&\geq 
	1 - K (\ln n) \exp \left(  - \frac{t_{1}}{4}  \right) \\
	&\geq
	1 - \exp \left( \ln( K \ln n )  -   \frac{t_{1}}{4} \right) \\
	&\geq
	1 - \exp \left( -\frac{t_{1}}{5}  \right),
	\end{align*}
	where the last inequality is valid since $t_{1}=\Omega(\ln n)$. 
\end{proof}

\subsection{Part III}\label{part3}

Finally we prove that $\mathbf{G}^{*}:= G^{(1)}\cup G^{(2)} \cup G^{(3)}$ percolates whp.

\begin{lemma}\label{parte3}
	Conditioned on $G^{(1)} \cup G^{(2)}$ containing a percolating subset $X$ 
	of size at least $n/2^{r+2},$ $\mathbf{G}^{*}$ percolates whp.
\end{lemma}

Indeed, we will prove that whp every vertex in $V\setminus X$ is connected to $X$ 
by edges of every colour by using the final round of exposure $G^{(3)}.$ 

\begin{proof}[\textbf{Proof of Lemma~\ref{parte3}}]
	
	We begin by  defining $\mathcal{K}$ to be the event that there is at least one vertex  $v \in V \setminus X$
	and one colour $i\in [r]$ such that $E_i^{(3)}$ contains no edge between $v$ and $X$.	
	Thus
	\begin{align*}
	\mathbb{P}[\mathcal{K}]
	\leq
	\sum_{i \in [r]} \sum_{v \in V \setminus X} \left(1-\frac{p_{i}}{3}\right)^{|X|}  
	\leq
	rn\left(1-\frac{p_{1}}{3}\right)^{\frac{n}{2^{r+2}}} 
	\leq
	rne^{-\frac{C_{r}}{3\cdot 2^{r+2}} \ln n }
	=
	rn^{1-\frac{C_{r}}{3\cdot 2^{r+2}}}
	\leq
	\frac{r}{n}
	=o(1),
	\end{align*}
	where the last inequality holds since $C_{r}\ge 2^{8r^2}\ge 3\cdot 2^{r+3}.$ 
	 Since $\bar{\mathcal{K}}$ implies that $\mathbf{G}^{\ast}$ percolates, 
	 this completes the argument.
\end{proof}

\section{Concluding remarks.}\label{concluding}

A number of open questions naturally present themselves.

\subsection{Optimising $C_r$}
Similar to Bollob\'as, Riordan, Slivken and Smith~\cite{jigsawBRSS},
we made no attempt to optimise the constant $C_r$ in
Theorems~\ref{MainResultConst} and~\ref{MainResultNonConst}. As a result,
the bounds on $P_r$ for the subcritical and supercritical case are a long
way apart. It is natural to expect them to be asymptotically
equal, leading to the following strengthening of Theorem~\ref{MainResultConst}:

\begin{conj}\label{conj:exactthreshold}
	Let $r \in \mathbb{N}$. There exists constants $C_1^*,C_2^*,\ldots, C_r^*$ 
	such that the following holds: suppose that $p_{1},...,p_{r}$ are functions of $n$ 
	such that $ 0 \leq p_{1}\leq p_{2} \leq ... \leq p_{r}\leq 1$
	and $\mathbf{G} = \mathbf{G}(n, p_{1}, p_{2} ,..., p_{r})$. For $i\in [r]$ let
	\textit{$P_{i}$ }$:= p_{1}p_{2}...p_{i}$. Then for any constant $\eps>0$:
	\begin{itemize}
		\item[$(i)$] If $P_{i} \leq \frac{(1-\eps)C_i^*}{n (\ln n)^{i-1}}$ for some $2 \leq i \leq r$ or
		$P_{1} \leq \frac{(1-\eps)\ln n}{n} $ then whp $\mathbf{G}$ does \emph{not} percolate.
		\item[$(ii)$]If $P_{i} \geq \frac{(1+\eps)C_i^*}{n (\ln n)^{i-1}}$ for every $ 2\leq i\leq r $ and
		$P_{1} \geq \frac{(1+\eps) \ln n}{n},$ then whp  $\mathbf{G}$ percolates.
	\end{itemize}
\end{conj}

It would be interesting to determine the exact value of the $C_i^*$.

\subsection{Size of the critical window}

If the $C_i^*$ can be determined precisely, the next parameter to optimise
would be the parameter $\eps$ in Conjecture~\ref{conj:exactthreshold}. More precisely,
does the result still hold if rather than $\eps$ being a constant it is allowed
to be a function of $n$ which tends to $0$ sufficiently slowly.
This has already been extensively studied in the case $r=1$, which
corresponds to connectedness of the graph, but is an open problem in general.

\subsection{Speed of the jigsaw process}

In the supercritical case of Theorem~\ref{MainResultConst}, we know that whp
the jigsaw percolation algorithm will terminate with just one cluster,
but how many steps does this process require?

More precisely, in each step we create an auxiliary graph on the clusters
of vertices, with an edge between clusters if there are edges between them
of every colour in the $r$-fold graph, and merge each connected component of this
auxiliary graph. How many iterations of this process are required before we have
one single remaining cluster?

An analysis of the proof shows that, for the random graphs considered in the
supercritical case, whp at most
$(1+o(1))(\ln n)^{1+1/r}$ steps are required.
However, this was not optimised and it would be natural to conjecture that actually
$\Theta(\ln n)$ steps are sufficient. It would also be interesting to determine the
constant in this $\Theta(\ln n)$ term, which would most likely be dependent on
how close the probability product $P_r$ is to the jigsaw percolation threshold.

\section{Acknowledgements}

We would like to thank Christoph Koch for his helpful comments on an earlier version of this paper,
including pointing out the simple proof of Claim~\ref{claim:alphabound}.

\end{document}